\documentclass{amsart}

\usepackage{amssymb,amscd,MnSymbol}
\usepackage{stmaryrd}
\usepackage{mathrsfs}
\usepackage{graphicx,psfrag}
\usepackage{color}

\usepackage{color}

\usepackage{tikz-cd, tikz}
\usetikzlibrary{patterns, matrix,shapes,arrows,calc,topaths,intersections,positioning,decorations.pathreplacing,decorations.pathmorphing,fit}

\usepackage[normalem]{ulem}
\newcommand{\congto}{\xrightarrow{\sim}}

\newcommand{\wt}{\widetilde}
\newcommand{\wb}{\overline}
\newcommand{\pa}{\partial}
\newcommand{\ot}{\otimes}
\newcommand{\C}{\mathbb{C}}
\newcommand{\Z}{\mathbb{Z}}
\newcommand{\CS}{{\C^*}}
\newcommand{\R}{\mathbb{R}}
\newcommand{\RM}{\backslash}
\newcommand{\hcal}{\mathcal{H}}
\newcommand{\beq}{\begin{equation} }
\newcommand{\eeq}{\end{equation} }
\newcommand{\bl}{\begin{lemma}}
\newcommand{\el}{\end{lemma}}
\newcommand{\bpp}{\begin{proposition}}
\newcommand{\epp}{\end{proposition}}
\newcommand{\bpf}{\begin{proof}}
\newcommand{\epf}{\end{proof}}
\newcommand{\bcr}{\begin{corollary}}
\newcommand{\ecr}{\end{corollary}}
\newcommand{\bd}{\begin{definition}}
\newcommand{\ed}{\end{definition}}
\newcommand{\bnum}{\begin{enumerate}}
\newcommand{\enum}{\end{enumerate}}

\DeclareMathOperator{\Log}{Log}
\DeclareMathOperator{\Skel}{Skel}

\newcommand{\la}{\langle}
\newcommand{\ra}{\rangle}
\newcommand{\In}{\subset}
\renewcommand{\ss}{\subsection}
\newcommand{\sss}{\subsubsection}

\newcommand{\bcs}{\begin{cases}}
\newcommand{\ecs}{\end{cases}}
\newcommand{\brem}{\begin{remark}}
\newcommand{\erem}{\end{remark}}

\newcommand{\ccal}{\cC}
\newcommand{\wcal}{\cW}
\newcommand{\La}{\Lambda}
\newcommand{\linf}{\Lambda^\infty}

\usepackage[initials]{amsrefs}

\newcommand{\si}{\sigma}
\newcommand{\Si}{\Sigma}

\newcommand{\bC}{\mathbb{C}}

\newcommand{\bF}{\mathbb{F}}
\newcommand{\bH}{{\mathbb{H}}}
\newcommand{\bK}{\mathbb{K}}
\newcommand{\bL}{\mathbb{L}}
\newcommand{\bP}{\mathbb{P}}
\newcommand{\bQ}{\mathbb{Q}}
\newcommand{\bR}{\mathbb{R}}
\newcommand{\bT}{\mathbb{T}}
\newcommand{\bZ}{\mathbb{Z}}

\newcommand{\boldA}{{\boldsymbol{A}}}

\newcommand{\bPsi}{{\boldsymbol{\Psi}}}

\newcommand{\bi}{{\bold{i}}}

\newcommand{\cA}{\mathcal{A}}

\newcommand{\cC}{\mathcal{C}}
\newcommand{\cD}{\mathcal{D}}

\newcommand{\cF}{\mathcal{F}}

\newcommand{\cM}{\mathcal{M}}
\newcommand{\cO}{\mathcal{O}}

\newcommand{\cU}{\mathcal{U}}

\newcommand{\cX}{\mathcal{X}}

\newcommand{\cW}{\mathcal{W}}
\newcommand{\cWL}{{\cW_\Lambda(T^*M_T)}}
\newcommand{\cWLo}{{\cWL^{\mathrm{op}}}}
\newcommand{\cWtL}{{\cW_{\tilde \Lambda}}}

\newcommand{\cWFS}{{\cW_{\mathrm{FS}}(T^*M_T;W_q)}}

\newcommand{\cY}{\mathcal{Y}}
\newcommand{\cZ}{\mathcal{Z}}

\newcommand{\fp}{\mathfrak{p}}

\newcommand{\ch}{\mathrm{ch}}
\newcommand{\Ch}{{\mathrm{Ch}}}
\newcommand{\crv}{{\mathsf{cr}}}

\newcommand{\Hom}{\mathrm{Hom}}
\newcommand{\Coh}{\mathrm{Coh}}

\newcommand{\eff}{{\mathrm{eff}}}

\newcommand{\ev}{\mathrm{ev}}

\newcommand{\Pic}{ {\mathrm{Pic}} }

\newcommand{\Perf}{{\mathrm{Perf}}}
\newcommand{\Ker}{\mathrm{Ker}}

\newcommand{\Sh}{\mathrm{Sh}}
\newcommand{\ShL}{{\mathrm{Sh}_\Lambda(M_T)}}
\newcommand{\ShLc}{{\mathrm{Sh}^w_\Lambda(M_T)}}
\newcommand{\ShtL}{{\mathrm{Sh}_{\tilde \Lambda}(M_\bR)}}
\newcommand{\ShtLc}{\mathrm{Sh}^w_{\tilde \Lambda}(M_\bR)}
\newcommand{\Spec}{\mathrm{Spec}}

\newcommand{\vir}{{\mathrm{vir}}}
\newcommand{\CC}{\mathrm{CC}}

\newcommand{\Nef}{ {\mathrm{Nef}} }
\newcommand{\NE}{ {\mathrm{NE}} }

\newcommand{\Jac}{{\mathrm{Jac}}}
\newcommand{\Hess}{{\mathrm{Hess}}}
\newcommand{\mir}{{\mathsf{mir}}}
\newcommand{\Mir}{{\mathsf{Mir}}}

\renewcommand{\Re}{{\mathrm{Re}}}
\renewcommand{\Im}{{\mathrm{Im}}}

\newcommand{\btau}{{\boldsymbol\tau }}

\newcommand{\fr}{\mathfrak{r}}

\newcommand{\fs}{\mathfrak{s}}

\renewcommand{\sf}{\mathsf{f}}

\newcommand{\tF}{{\tilde{F}}}
\newcommand{\tbT}{ {\tilde{\bT}} }
\newcommand{\tcD}{ {\tilde{\mathcal{D}}}}

\newcommand{\tb}{{\tilde{b}}}

\newcommand{\tL}{{\tilde{L}}}
\newcommand{\tM}{{\tilde{M}}}
\newcommand{\tN}{{\tilde{N}}}

\newcommand{\tX}{\tilde{X}}

\newcommand{\tcY}{\tilde{\cY}}

\newcommand{\lra}{\longrightarrow}

\newcommand{\Ue}{{U_{\epsilon}}}
\newcommand{\Uec}{{U_{\epsilon}^\circ}}
\newcommand{\tUe}{{\tilde U_{\epsilon}}}
\newcommand{\tUec}{{\tilde U_{\epsilon}^\circ}}
\newcommand{\Uee}{{U_{\epsilon,\epsilon'}}}
\newcommand{\Ueec}{{U_{\epsilon,\epsilon'}^\circ}}
\newcommand{\tUee}{{\tilde U_{\epsilon,\epsilon'}}}
\newcommand{\tUeec}{{\tilde U_{\epsilon,\epsilon'}^\circ}}

\newtheorem*{Theorem*}{Theorem}
\newtheorem{dummy}{dummy}[section]
\newtheorem{lemma}[dummy]{Lemma}
\newtheorem{theorem}[dummy]{Theorem}
\newtheorem{conjecture}[dummy]{Conjecture}
\newtheorem{corollary}[dummy]{Corollary}
\newtheorem{proposition}[dummy]{Proposition}

\theoremstyle{definition}
\newtheorem{definition}[dummy]{Definition}
\newtheorem{remark}[dummy]{Remark}

\newtheorem{assumption}[dummy]{Assumption}

\begin{document}
\title[Gamma II from HMS]{Gamma II for toric varieties from integrals on T-dual branes and homological mirror symmetry}

\address{Bohan Fang, Beijing International Center for Mathematical
  Research, Peking University, 5 Yiheyuan Road, Beijing 100871, China}
\email{bohanfang@gmail.com}
\author{Bohan Fang}

\address{Peng Zhou, Institut des Hautes \'Etudes Scientifiques. Le Bois-Marie, 35 route de Chartres, 91440 Bures-sur-Yvette France}
\email{pzhou.math@gmail.com}
\author{Peng Zhou}

\begin{abstract}
In this paper we consider the oscillatory integrals on Lefschetz thimbles in the Landau-Ginzburg model as the mirror of a toric Fano manifold. We show these thimbles represent the same relative homology classes as the characteristic cycles of the corresponding constructible sheaves under the equivalence of \cite{GPS18-2}. Then the oscillatory integrals on such thimbles are the same as the integrals on the characteristic cycles and relate to genus $0$ Gromov-Witten descendant potential for $X$, and this leads to a proof of Gamma II conjecture for toric Fano manifolds.
\end{abstract}
\maketitle

\section{Introduction}

The mirror of a toric Fano variety $X$ is a Landau-Ginzburg model $W:(\bC^*)^{\dim X} \to \bC$ where \emph{superpotential} $W$ is a Laurent polynomial.

The \emph{closed} A-model of $X$, mathematically, is about its Gromov-Witten theory. By mirror symmetry it can be read from the \emph{closed} B-model on the mirror Landau-Ginzburg model, usually in the form of period integrals. In principle, genus $0$ Gromov-Witten descendant potential function of $X$ is equal to oscillatory integrals on its mirror \cites{Gi94}
\[
\int e^{-\frac{W}{z}}\frac{dX_1\dots dX_n}{X_1\dots X_n}.
\]

A natural question to ask is that the mirror effect on the Gromov-Witten side for the choice of cycles one integrates over. The answer is given by identification of the K-group $K(X)$ with relative homology cycles of $H_n((\bC^*)^n, \Re(W/z)\gg 0)$ in \cites{Iritani09}. By inserting certain Gamma-function related characteristic class of such $K$-group element in the Gromov-Witten correlator function, then one can show this is equal to the oscillatory integral over such relative cycle. Moreover in \cites{Fang16} this identification is shown to agree with homological mirror symmetry: a categorical identification of coherent sheaves on $X$ and a Fukaya-type category of Lagrangian cycles in $(\bC^*)^n$, which relates the \emph{open} {B-model} on $X$ to the \emph{open} A-model on its mirror Landau-Ginzburg model.

There are various consequences of such mirror symmetry. Since genus $0$ descendant correlator functions are solutions to quantum differential equations, one can investigate the properties of such solutions by analytic methods on the oscillatory integrals. A particularly interesting type of relative cycles for the Landau-Ginzburg model are Lefschetz thimbles. On one hand they correspond to a full exceptional collection in any reasonbly-defined Fukaya-type category associated to $((\bC^*)^n,W)$. On the other hand the oscillatory integrals over them have nice asymptotic properties. The Gamma conjectures \cites{GGI16} for Fano varieties, especially the Gamma II conjecture, are related to these features.

\subsection{Gamma conjectures for Fano varieties}

The Gamma conjecture \cites{GGI16} for Fano varieties is about its quantum cohomology and its certain characteristic classes. Let $X$ be a Fano variety and we define the limit of the $J$-function (assembled from some Gromov-Witten genus $0$ descendant invariants) on a ray with coordinate $t>0$ in its K\"ahler cone (under certain assumption, namely \emph{property $\mathcal O$})
\[
A_X=\lim_{t\to +\infty} \frac{J_X(t)}{\langle [\mathrm{pt}],J_X(t)\rangle} \in H^*(X).
\]
The \emph{Gamma I conjecture} says this class is the \emph{Gamma class} of its tangent bundle $\hat \Gamma_X$ (see Equation \eqref{eqn:Gamma}) \cites{Iritani09,KKP08}.

Genus $0$ Gromov-Witten theory defines the (small) quantum connection on the trivial $H^*(X)$ bundle over $H^2(X)\times \bC^*$  (Equation \eqref{eqn:qde-tau}), which can be solved asymptotically \cites{Du93,Gi01a} (Theorem \ref{thm:Dubrovin-Givental-decomposition}). These asymptotic solutions $y_1,\dots,y_\fs$ form a basis of soutions. On the other hand, any solution is described by (linearly combination of) certain genus $0$ Gromov-Witten correlator function $\cZ(E)$ for $E\in K(X)$, whose definition involves inserting $\boldA_E=\Ch(E)\cdot \hat\Gamma_X$, where $i=1,\dots,\fs=\dim H^*(X)$. The Gamma II conjecture says the following.
\begin{conjecture}[Gamma II, see Conjecture \ref{conj:gamma-ii} for its precise form, and Theorem \ref{thm:gamma-ii} for complete toric Fano manifolds]
   There exists a full exceptional collection $E_1,\dots, E_\fs\in \Coh(X)$
    such that
   \[
   \cZ([E_i])=y_i.
   \]
\end{conjecture}
These $A_i:=\boldA_{[E_i]}$ are called \emph{higher asymptotic classes}. The Gamma I conjecture has been proved for complex Grassmannians \cite{GGI16}, Fano 3-folds of Picard rank one \cite{GoZa16}, Fano complete intersections in projective spaces \cites{GaIr15, SaSh17, Ke18}, toric Fano manifolds that satisfy B-model analogue of Property $\cO$ \cites{GaIr15}, and del Pezzo surfaces \cites{HuKeLiYang19}.

Gamma II conjecture is an extension of Dubrovin's conjecture \cites{Dubrovin1998} and formulated in \cites{Dubrovin13, GGI16}. It is known for projective spaces \cites{GaIr15}. A $K$-theoretic version of the Gamma II conjecture is also shown for complete toric Fano manifolds \cites{GaIr15}.

We would like to remark that Gamma conjecture is motivated by mirror symmetry, and has a version more directly related to Strominger-Yau-Zaslow conjecture for Calabi-Yaus \cites{AGIS18}. In particular the Gamma class arises naturally in the B-model period integral computation \cites{HLY96,Ho06}.

\subsection{Oscillatory integrals and mirror symmetry}

In some situation, the Gamma conjecture can be mathematically proved by mirror symmetry (see \cite{GaIr15} for various cases). This paper considers a smooth toric Fano manifold $X$ with Landau-Ginzburg mirror $W:(\mathbb C^*)^n\to \bC$. In \cites{Iritani09}, it is shown that the oscillatory integrals over integral cycles in $H_n((\mathbb C^*)^n, \mathrm{Re}(W/z)\gg 0)$ are given by genus $0$ descendant potential with classes $\boldA_{[E]}$ inserted -- the lattice of such relative homology cycles is isomorphic to the $K$-group lattice of holomorphic vector bundles $E$ on $X$. In \cite{Fang16} such isomorphism is further shown to agree with homological mirror symmetry in the sense of \cites{FLTZ11, FLTZ12,Kuwagaki17,ZhouPeng17b,Vaintrob16}. Such HMS (or more precisely \emph{coherent-constructible correspondence}) say for any coherent sheaf $E$ on $X$ one can associates a constructible sheaf on $(S^1)^n$, and the category of coherent sheaves $\Coh(X)$ is derived equivalent to the category of such constructible sheaves on $(S^1)^n$, denoted by $\Sh^w_\Lambda((S^1)^n)$. Iritani's isomorphism of the $K$-group and the relative homology for the LG model is obtained by taking a coherent sheaf $E$'s corresponding constructible sheaf on $(S^1)^1$ and then taking its characteristic cycle, which is a Lagrangian cycle in $(\mathbb C^*)^n\cong T^* (S^1)^n$ and represents a class in $H_n((\mathbb C^*)^n,\mathrm{Re}(W/z)\gg 0)$.

In this paper we want to consider Lagrangian thimbles associated to the Landau-Ginzburg mirror $W:(\mathbb C^*)^n\to \bC$. Their images on $W$ are right-pointing rays, and thus represent classes in $H_n((\mathbb C^*)^n,\mathrm{Re}(W/z)\gg 0)$ (for $\mathrm{Re}(z)>0$). They are naturally objects in the Fukaya-Seidel category of the Landau-Ginzburg model \cite{Seidelbook}. We use the recently developed Ganatra-Pardon-Shende's wrapped Fukaya category \cite{GPS17,GPS18-1} $\cWFS$ instead of the original version \cite{Seidelbook} as the Fukaya-Seidel category of the Landau-Ginzburg model (see Section \ref{sec:HMS-Lag} for the notion). Then we have the following.
\begin{itemize}
  \item These thimbles form a full exceptional collection.
  \item By the recent result of \cite{GPS18-2}, they correspond to constructible sheaves on $(S^1)^n$.
  \item By \cite{ZhouPeng18, GammageShende17}, such constructible sheaves are in $\Sh^w_\Lambda((S^1)^n)$.
\end{itemize}
We further show that such thimbles represent the same relative homology classes as the characteristic cycles of their corresponding constructible sheaves.

Then we have a natural pathway to Gamma II conjecture: passing to constructible sheaves and then to coherent sheaves on $X$ we have a full exceptional collection $E_1,\dots, E_n$. On one hand we analyze the asymptotic behavior of the oscillatory integrals on these thimbles, which are precisely asymptotic solutions $y_i$. On the other hand they are integrals over characteristic cycles of corresponding constructible sheaves and thus are equal to $1$-point Gromov-Witten descendant potentials with $\boldA_{[E_i]}$ inserted, and thus $y_i=\mathcal Z([E_i])$.

\begin{remark}
  We show the Gamma II conjecture in a neighborhood of the large radius limit (complex parameter $|q|\ll 1$). Actually the validity of Gamma II at any semisimple point implies the rest (see the proof of Theorem 6.4 \cite{GaIr15}.)
\end{remark}

\subsection{Outline}

We recall the notion of a smooth toric Fano variety $X$ and its mirror in Section \ref{sec:mirror}. In particular we very carefully define the mirror $(\mathbb C^*)^n$ as a complex manifold (in Section \ref{sec:LG-B-model}) and as a symplectic manifold $T^*(S^1)^n$ (in Section \ref{sec:lg-a}). A key ingredient is the identification of both which is explained in Section \ref{sec:lg-a}. Then we define closed-sector theory in Section \ref{sec:thimbles} for both B-side (oscillatory integrals) and A-side (descendant Gromov-Witten invariants and quantum connection on $X$). In Section \ref{sec:HMS-sheaf} we recall the coherent-constructible correspondence in \cites{FLTZ11,Kuwagaki17,ZhouPeng17b,Vaintrob16} first and show the convergence of oscillatory integration on characteristic cycles. In Section \ref{sec:HMS-Lag} we show that in Ganatra-Pardon-Shende's wrapped Fukaya category \cites{GPS17,GPS18-1}, a Lefstchez thimble and or a ``standard'' Lagrangians represent the same relative homology classes as its corresponding constructible sheaf's (by \cites{GPS18-2}) characteristic cycle. Then in Section \ref{sec:gamma-ii} we show the Gamma II conjecture for a toric Fano manifold by looking at the oscillatory integral on Lefschetz thimbles, and its mirror version: genus $0$ descendant GW potential with asymptotic classes of the mirror coherent sheaves in an exceptional collection.

\subsection{Acknowledgements}

BF would like to thank Hiroshi Iritani for bringing this probem to attention. He is also grateful to the very helpful discussion with Chiu-Chu Melissa Liu, David Nadler, Vivek Shende and Eric Zaslow.

The work of BF is partially support by an NSFC grant 11831017. The work of PZ is supported by an IHES Simons Postdoctoral Fellowship as part of the Simons Collaboration on HMS.

\section{Mirror symmetry for toric manifolds}
\label{sec:mirror}

In this section, we fix the notion of toric manifolds and discuss their mirror Landau-Ginzburg A and B-models.

\subsection{Definition of a toric manifold}
Let $N\cong \bZ^n$ be a finitely generated free abelian group, and let $N_\bR=N\otimes_\bZ\bR$. We consider complete smooth toric manifolds given by a simplicial fan $\Si$ in $N_\bR$ such that the set of $1$-cones is $$
\{\rho_1,\dots,\rho_{\fr}\},
$$
where $\rho_i\cap N=\bZ_{\ge 0} b_i$, $i=1,\dots, \fr$. We require
\begin{itemize}
  \item $\Si$ is complete: $|\Si|=N_\bR$;
  \item $\Si$ is smooth: for every top dimensional cone $\si$, the lattice $\oplus_{b_i\in\si}\bZ b_i\cong N$.
\end{itemize}

There is a surjective group homomorphism
\begin{eqnarray*}
\phi: & \tN :=\oplus_{i=1}^r \bZ\tb_i & \lra  N,\\
        &  \tb_i  & \mapsto b_i.
\end{eqnarray*}
Define $\bL :=\Ker(\phi) \cong \bZ^\fp$ where $\fp:=\fr-n$. Then
we have the following short exact sequence of finitely generated abelian groups:
\begin{equation}\label{eqn:NtN}
0\to \bL  \stackrel{\psi }{\lra} \tN  \stackrel{\phi}{\lra} N\to 0.
\end{equation}
Applying $ - \otimes_\bZ \bC^*$ and $\Hom(-,\bZ)$ to \eqref{eqn:NtN}, we obtain two exact
sequences of abelian groups:
\begin{align}
\label{eqn:bT}
&1 \to G \to \tbT\to \bT \to 1,\\
&\label{eqn:MtM}
0 \to M \stackrel{\phi^\vee}{\to} \tM \stackrel{\psi^\vee}{\to} \bL^\vee \to 0,
\end{align}
where
\begin{align*}
&\bT= {N}\otimes_\bZ \bC^* \cong (\bC^*)^n,\  \tbT = \tN\otimes_\bZ \bC^* \cong (\bC^*)^\fr,\  G = \bL\otimes_\bZ \bC^* \cong (\bC^*)^\fp,\\
&M = \Hom(N,\bZ)  = \Hom(\bT,\bC^*), \
\tM = \Hom(\tN,\bZ)= \Hom(\tbT,\bC^*),\
\bL^\vee = \Hom(\bL,\bZ) =\Hom(G,\bC^*).
\end{align*}

The action of $\tbT$ on itself extends to a $\tbT$-action on $\bC^\fr = \Spec\bC[Z_1,\dots, Z_\fr]$.
The group $G$ acts on $\bC^\fr$ via the group homomorphism $G\to \tbT$ in \eqref{eqn:bT}.

Define the set of ``anti-cones''
$$
\cA=\{I\subset \{1,\dots, \fr\}: \text{$\sum_{i\notin I} \bR_{\ge 0} b_i$ is a cone of $\Si$}\}.
$$
Given $I\in \cA$, let $\bC^I$ be the subvariety of $\bC^\fr$ defined by the ideal in $\bC[Z_1,\ldots, Z_\fr]$ generated by $\{ Z_i \mid i\in I\}$.  Define the toric orbifold $X$ as the stack quotient
$$
X:=U_\cA/ G,
$$
where
$$
U_\cA:=\bC^\fr \backslash \bigcup_{I \in \mathcal A} \bC^I.
$$
The smooth compact variety $X$ contains the torus $\bT:= \tbT/G$ as a dense open subset, and
the $\tbT$-action on $\cU_\cA$ descends to a $\bT$-action on $X$.

Let $\tcD_i$ be the $\tbT$-divisor in $\bC^\fr$ defined by $Z_i=0$. Then
$\tcD_i \cap \cU_A$ descends to a $\bT$-divisor $\cD_i$ in $\cX$. We have
$$
\tM \cong \Pic_{\tbT}(\bC^\fr) \cong H^2_{\tbT}(\bC^\fr;\bZ),
$$
where the second isomorphism is given by the $\tbT$-equivariant
first Chern class $(c_1)_{\tbT}$.
 Define
\begin{gather*}
D_i=(c_1)_G(\cO_{\bC^\fr}(\tcD_i)) \in H^2_G(\bC^r;\bZ)\cong \bL^\vee
\end{gather*}
We have
$$
\Pic(X)\cong H^2(X;\bZ) \cong \bL^\vee.
$$

\subsection{The nef and Mori cone} \label{sec:nef-NE}
In this paragraph, $\bF=\bQ$, $\bR$, or $\bC$.
Given a finitely generated free abelian group $\Gamma \cong \bZ^m$, define
$\Gamma_\bF= \Lambda\otimes_\bZ \bF \cong \bF^m$.
We have the following short exact sequences of vector spaces ($\otimes_\bZ
\bF$ with Equation \eqref{eqn:NtN} and \eqref{eqn:MtM}):
\begin{eqnarray*}
&& 0\to \bL_\bF\to \tN_\bF \to N_\bF \to 0,\\
&& 0\to  M_\bF\to \tM_\bF\to \bL^\vee_\bF\to 0.
\end{eqnarray*}

Let $\Si(d)$ be the set of $d$-dimensional cones. For each $\si\in \Si(d)$. Given a maximal cone $\si\in \Si(n)$, we define the nef cone $\Nef_X$ as below
$$
\Nef_\si = \sum_{i\in I_\si}\bR_{\geq 0} D_i,\quad \Nef_{X}:=\bigcap_{\si\in \Si(n)} \Nef_\si.
$$
The $\si$-K\"{a}hler cone $C_\si$ is defined to be the interior of $\Nef_\si$; the  K\"{a}hler cone of $X$, $C_{X}$, is defined to be the interior of the  nef cone $\Nef_{X}$.

Let $\langle-, -\rangle$ be the natural pairing between
$\bL^\vee_\bQ$ and $\bL_\bQ$. We define the Mori cone $\NE_\si\subset \bL_\bR$ to be
$$
\NE_{X}:= \bigcup_{\si\in \Si(n)} \NE_\si,\quad  \NE_\si=\{ \beta \in \bL_\bR\mid \langle D,\beta\rangle \geq 0 \ \forall D\in \Nef_\si\}.
$$
Finally, we define curve classes
$$
\bK_{\eff,\si}:= \bL \cap \NE_\si,\quad \bK_{\eff}:=
\bL\cap \NE_X.
$$

\begin{assumption}[Fano condition] \label{semi-positive}
From now on, we assume $D_1+\dots + D_\fr$ is contained in the K\"ahler cone $C_X$, which is equivalent to $c_1(X)>0$, i.e. $X$ is a Fano variety.
\end{assumption}


\subsection{Landau-Ginzburg as the B-model}
\label{sec:LG-B-model}

In this subsection, we define the mirror Landau-Ginzburg model from the viewpoint of complex geometry, and identify it with
$T^*{M_\bR}=N_\bR\times M_\bR$.

We fix an integral basis $e_1,\dots,e_\fp \in \bL$ and its dual basis $e_1^\vee,\dots,e_\fp^\vee$ in $\bL^\vee$. We require that each $e_a^\vee$ is in $\Nef_X$. As discussed in \cite[p1037]{Iritani09}, this choice is always possible. We let $H_1,\dots, H_\fs$ be a $\bZ$-basis of $H^*(X;\bZ)$, and $H_a=e_a^\vee$ for $a=1,\dots,\fp$. Here $\fs=\dim H^*(X)$.

Define the \emph{charge vectors}
\[
l^{(a)}=(l_1^{(a)},\dots, l_\fr^{(a)})\in \bZ^\fr,\quad
\psi(e_a)=\sum_{i=1}^\fr l_i^{(a)} \tb_i.
\]
So
\[
D_i=\psi^\vee(D_i^\bT)=\sum_{a=1}^\fp l_i^{(a)}e_a^\vee,\quad i=1,\dots,\fp.
\]
Define the Landau-Ginzburg B-model as follows
$$
\cY_q=\{(\tX_1,\dots,\tX_\fr)\in (\bC^*)^\fr|\prod_{i=1}^\fr \tX_i^{l^{(a)}_i}=q_a,\ a=1,\dots,\fp\}.
$$
Here $q_1,\dots,q_\fr$ are \emph{complex parameters}. Apply the exact functor $\Hom(-,\bC^*)$ to the short exact sequence
\eqref{eqn:NtN} and we get
\[
1\to \Hom(N,\bC^*) \to (\bC^*)^\fr \stackrel{\mathfrak{q}}{\longrightarrow} \cM=\Hom(\bL,\bC^*)\to 1.
\]
We see that $\cY_q=\mathfrak{q}^{-1}(q)\cong (\bC^*)^k$ is a subtorus in $(\bC^*)^\fr$. Here $q=(q_1,\dots, q_\fp)$ are coordinates on $\cM$. For any $\beta\in \bL$, denote $q^\beta=\prod_{a=1}^\fp q_a^{\langle \beta, e_a^\vee \rangle}$.

Let $u_1,\dots,u_n$ and $u'_1,\dots,u'_n$ be the two sets of coordinates on $M_\bR$. Let $y_i=-v_i+2\pi \bi u_i$ and
$Y_i=e^{y_i}$. Then $y_1,\dots,y_n$ are complex coordinates on
$\tcY=M_\bR\times M_\bR\cong \bC^n$, while $Y_1,\dots,Y_n$ are complex coordinates on
$\cY= M_{\bC^*}=M_\bR/M\times M_\bR\cong T(M_\bR/M)\cong (\bC^*)^n$.

We fix a splitting of the exact sequence \eqref{eqn:NtN}, i.e. we choose a surjective map $\eta: \tN \to \bL$ such that
$\eta(\tb_i)=\sum_{a=1}^\fp \eta_{ia} e_a$ (so $e^\vee_a=\sum_{i=1}^\fr
\eta_{ia}D_i$) and $\psi\circ\eta=\mathrm{id}$, where $\eta_{ia}\in \bZ$. This splitting identifies
$\cY_q$ with $\cY=M_{\bC^*}=\Hom(N,\bC^*)$
\begin{equation}
\label{eqn:B-model-identification}
X_i=q'_iY^{b_i},\quad Y^{b_i}=\prod_{j=1}^n Y_j^{b_{ij}},\quad q_i'=\prod_{a=1}^\fp q_a^{\eta_{ia}}.
\end{equation}
Here $b_i=(b_{i1},\dots,b_{in})$ is the coordinate of $b_i$ in
$N$. We also identifies $\tcY_q$ with $\tcY=M_\bR \times M_\bR$. The splitting $\eta$ specifies an isomorphism $\tN=\bL\oplus N$ and $\tM=\bL^\vee \oplus M$.

The superpotential on $\cY_q$ is
$$
W=\sum_{i=1}^\fr \tX_i.
$$
The following holomorphic form on $\cY=M_\bR/M \times M_\bR=TM_T$
$$
\Omega=\frac{dY_1\dots dY_n}{Y_1\dots Y_n}.
$$
Here we denote $M_T=M_\bR/M\cong (S^1)^n$. Let $W_\eta$ be the function $W$ on $\cY$ once we identify $\cY_q$ with $\cY$ by $\eta$ via Equation \eqref{eqn:B-model-identification}. We still denote $W=W_\eta$ as a holomorphic function on $Y$, keeping in mind the choice $\eta$ we have made. We will see later it does not play any role in the computation of integral, and even this fact is not directly needed in the proof of the main theorem of this paper (Theorem \ref{thm:gamma-ii}). The superpotential $W$ on $\cY$ is a Laurent polynomials in $X_1,\dots, X_n$ and $q_1,\dots, q_\fp$ and we denote it by $W_q$ when fixing $q_1,\dots, q_\fp$.

\subsection{Setup of LG A-model}

\label{sec:lg-a}

Recall that in the fan $\Sigma \In N_\R$, rays $\rho_i=\bR_{\geq 0} b_i\in \Sigma(1)$ for $i=1,\dots,\fr$, while $b_i$ are primitive generators of $\rho_i \cap N$. By the smoothness condition of our toric variety $X$, each top dimensional cone $\sigma \in \Sigma$ is simplicial,  and the ray generators $b_i\in\sigma$ forms a $\Z$-basis of $N$. Let $A = \{b_1, \cdots, b_\fr\}$, and $Q = \text{ConvHull}(A)$ be the convex hull of $A$ in $N_\R$.

By Equation \eqref{eqn:B-model-identification} the superpotential $W_q$ can thus be written as
\[ W_q(z) = \sum_{\alpha \in A}^{\fr} c_\alpha(q) X^{\alpha} \]
where $c_\alpha(q) \in \C^*$ depends on choice of $q \in \cM$. If we choose $q=1 \in \cM$, then $c_\alpha(q) = 1$.

We have a canonical Log map, $\Log: M_\CS \to M_\R$, induced by $\log|\cdot |: \C^* \to \R$. Mikhalkin shows that \cites{Mikhalkin04}, the image $\cA_t:= \frac{\Log(W_q^{-1}(t ))}{\log |t|}$ of a fiber of $W_q$ over $t \in \C$ under the rescaled $\Log$ map,  converges as $|t| \to \infty$ to a polyhedral complex $\Pi_A$ in $M_\R$, which only depends on $A \In N$ and is independent of $q$. The complements of $\Pi$ has a one-to-one correspondence with elements in  $A \sqcup \{0\}$, and this $0\in N$ corresponds to the compact polytope
\[ P = \{ x \in M_\R \mid \la x, \alpha \ra \leq 1 \} \subset M_\R.
\]
$P$ is also the dual polytope to $Q \In N_\R$. By the smooth and Fano condition, $P$ is also a lattice polytope.

We choose a continuous and homogeneous degree two convex function $\varphi_\R: M_\R \to \R$, which is smooth on $M_\R \RM 0$, such that each positive dimension face $F$ of $P$ has a minimum of $\varphi$ in the interior of $F$.\footnote{We may smooth $\varphi_\R$ near $0$, but it does not matter for the statement on skeleton for a fiber of $W_q$ near $\infty$. Furthermore, our identification $M_\bR$ and $N_\bR$ is later used to translate Lagrangian cycles in $T^*M_T$ to $\cY$ to perform oscillatory integrals on, which does not require the smoothness of this identification.} In \cites{ZhouPeng18}, the second-named author shows that such function $\varphi$ exists and has a contractible choice. Let $\varphi = \Log^{*}(\varphi_\R)$ be a Kahler potential on $M_\CS$, and
\[ \lambda_\varphi = - d^c \varphi, \quad \omega_\varphi = -dd^c\varphi \] be the Liouville one-form and symplectic two-form on $M_\CS$. If we fix an identification of $M_\CS \cong (\CS)^n$, with complex coordinates $z_i$ and polar coordinates $(\rho_i, \theta_i) \in \R \times S^1$, such that $z_i = e^{\rho_i + i \theta_i}$, then we have
\[ \lambda_\varphi = \sum_i \pa_i \varphi_\R(\rho) d \theta_i , \quad \omega_\varphi = \sum_{i,j} \pa_{ij} \varphi_\R(\rho) d\rho_i \wedge d \theta_i. \]
The Riemannian metric defined by $g_\varphi (X, Y) = \omega_\varphi(X, JY)$ is then
\[ g_\varphi = \sum_{i,j}  \pa_{ij} \varphi_\R(\rho)( d\rho_i \ot d\rho_j + d\theta_i \ot d\theta_j). \]

The above choice of $\varphi_\R: M_\R \to \R$ induces a Legendre transformation
\[ \Psi_\R: M_\R \congto N_\R, \quad \rho \mapsto - d \varphi_\R|_\rho \in T^*_\rho M_\R \cong N_\R. \]
The extra minus sign here is added for the purpose of integration which we will discuss in next section. Since $\varphi_\bR$ is convex and homogeneous degree $2$, $\Psi_\R$ sends rays in $M_\R$ to that in $N_\R$. Thus, using canonical isomorphism $M_\CS \cong M_T \times  M_\R$ and $T^*M_T \cong M_T \times N_\R$, we have
\beq \Psi: M_\CS  \congto T^*M_T. \label{legendre} \eeq
If we equip $T^*M_T \cong \{(\theta_i, p_i) \in T^n \times \R^n\} $ with the exact symplectic structure :
\[ \lambda_{std} = -\sum_i p_i d \theta_i, \quad \omega_{std} = \sum_{i} d \theta_i \wedge dp_i, \]
then one easily checks that,  $\Psi$ preserves the  Liouville structure
\[ \Psi^*(\lambda_{\mathrm{std}}) = \lambda_{\varphi}, \quad  \Psi^*(\omega_{\mathrm{std}}) = \omega_{\varphi}. \]


\section{Oscillatory integrals, GW invariants and Gamma class}

\subsection{Lattice generated by Lefschetz thimbles}
\label{sec:thimbles}

We define regions $\Ue$ and $\Uee$ for complex parameters $q_1,\dots, q_\fp$ as below
\begin{align*}
  &\Ue=\{|q_i|<\epsilon, i=1,\dots,\fp\}.\\
  &\Uee=\{|q_i|<\epsilon, -\epsilon'< \arg q_i<\epsilon, i=1,\dots,\fp\}.
\end{align*}
For any $q\in \cM$ we define
\[
\bH_q=H_n(\cY, \Re(W_q)\gg 0;\bZ).
\]
By \cite[Lemma 3.8 and Proposition 3.12]{Iritani09}, when $|q|\in \Ue$, $\bH_q\cong \bZ^\fs$ where $\fs=\dim H^*(X)$ for some small $\epsilon$. We choose such $\epsilon$ and requrie $q\in \Ue$. We denote $\Uec$ and $\Ueec$ to be an open subset of $\Ue$ and $\Uee$ such that $W_q$ is holomorphic Morse with distinct critical values respectively. They are open and dense subsets.
Let $\crv_1,\dots,\crv_\fs$ be such critical values.
\begin{definition}
An \emph{admissible phase $\theta$ of $W_q$} satisfies
\[
\Im(\crv_i e^{-\bi \theta})\neq \Im(\crv_j e^{-\bi \theta}),
\]
i.e. the segment between $\crv_i$ and $\crv_j$ is not parallel to $e^{\bi\theta}$. Here we denote $\bi=\sqrt{-1}$.
\end{definition}

\begin{definition}
  \label{def:thimbles-cycle}
We define the Lefschetz thimbles $\Gamma_i$ in phase $\theta$ to be the Leschetz thimbles associated to $\gamma_i$ where $\gamma_i$ is illustrated in the following figure.
\begin{figure}[h]
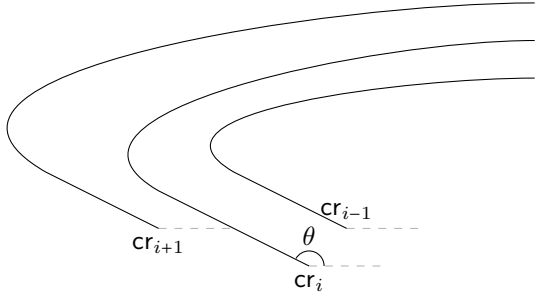

\centering{
\tikz{
\draw (0,0) to  +(-2,1) to[out=150, in =180]  (3,3);
\draw [dashed, opacity=0.3] (0, 0) to +(1,0);
\node [below] at (0,0) {$\mathsf{cr}_{i}$};
\draw (0.5, 0.5) to +(-1.5, 0.75) to[out=150, in =180]  (3,2.5);
\draw [dashed, opacity=0.3] (0.5, 0.5) to +(1,0);
\node [right, above] at (0.5,0.5) {$\mathsf{cr}_{i-1}$};
\draw (-2,0.5) to +(-1.5,0.75) to[out=150, in =180] (3, 3.5);
\draw [dashed, opacity=0.3] (-2, 0.5) to +(1,0);
\node [below] at (-2,0.5) {$\mathsf{cr}_{i+1}$};

\draw (0.2, 0) arc (0:150:0.2);
\node at (0,0.4) {$\theta$};
}
\caption{Vanishing paths and Fukaya-Seidel category}
\label{fig:topView}
}

\end{figure}


Each $\gamma_i$ starts from the critical value $\crv_i$ in a straight line of direction $\theta$, and turns to the direction of a very small positive argument $\delta$ (clockwisely of $\theta>\delta$ or counterclockwisely if $\theta<\delta$). It then becomes a ray in the direction of $e^{\bi \delta}$ and goes to $e^{\bi\delta}\infty$. Since $\theta$ is admissible, these $\gamma_i$ do not intersect each other.
\end{definition}

These Lefschetz thimbles $\Gamma_1,\dots, \Gamma_\fs$  are generators of $\bH_q$. For any cycle $\Gamma\in \bH_q$, we use
\[
\int_{\Gamma} e^{-\frac{W}{z}}\Omega
\]
to denote the integration of the differential form $e^{-\frac{W}{z}}$ on $\Gamma$ for any $\Re (z)>0$. 

There is a pairing in $\bH_q$. We may parallel translate $\Gamma\in H_n(\cY, \Re (W_q) \gg 0;\bZ)$ to $e^{\bi \pi} \Gamma\in H_n(\cY, \Re (W_q) \ll 0;\bZ)$ by isotope the class through $H_n(\cY, \Re ( e^{-\bi\theta}W_q) \gg 0;\bZ)$ for $\theta$ from $0$ to $\pi$, i.e. rotating the tail of the vanishing paths counter-clockwise by 180 degree. Then we define the pairing of $\Gamma,\Gamma'\in \bH_q$ by the signed count of the intersection number
\[
S(-,-): \bH_q \times \bH_q \to \bZ, \quad S(\Gamma, \Gamma') := \sharp(e^{\bi \pi} \Gamma \cap \Gamma').
\]
This is also the perfect pairing between $H_n(\cY,\Re(W_q)\gg 0;\bZ) $ with $H_n(\cY,\Re(W_q)\ll 0;\bZ)$. Define the \emph{Stokes matrix} of $(\cY,W)$ at phase $\theta$
\[
S_{ij}=S(\Gamma_i,\Gamma_j).
\]
In particular this is an upper triangular integral matrix (\cite[Corollary 4.12]{Iritani09}) \footnote{Our sign convention slightly differs from Iritani's, but $S_{ij}$ remains an upper triangular matrix as in there. }
\[ S_{ii}=1, \quad \text{and}\;\; S_{ij}=0 \, \forall i > j \]

\subsection{Genus-0 descendant potential and Iritani's theorem}

We fix some notions on Gromov-Witten theory.
\begin{definition}
  \label{def:gw-potential}
  Let $X$ be a complete toric manifold. We define genus $g$, degree $d\in H_2(\cX;\bZ)$, descendant Gromov-Witten invariants of $\cX$ as
  \[
    \langle \tau_{a_1}(\gamma_1)\dots \tau_{a_n}(\gamma_n)\rangle_{g,n,d}^{X}=\langle \gamma_1 \psi_1^{a_1}\dots \gamma_n\psi_n^{a_n}\rangle_{g,n,d}^{X}=\int_{[\overline{\cM}_{g,n}(X;d)]^\vir}\prod_{j=1}^n \psi_j^{a_j}\ev^*_j(\gamma_j)\in \bQ,
  \]
  where $\gamma_i\in H^*(X)$ and $\ev_j:\overline{\cM}(X;d)\to X$ is the $j$-th evaluation map. Let $\btau\in H^{2}(X;\bC)$. We also define
  \begin{align}
    \label{eqn:correlator}
   \llangle \tau_{a_1}(\gamma_1),\dots,\tau_{a_n}(\gamma_n)\rrangle_{g,n}^{\cX}=\llangle \gamma_1 \psi_1^{a_1},\dots,\gamma_n\psi_n^{a_n}\rrangle_{g,n}^{\cX}\\\nonumber
   =\sum_{d\in \bK_{\mathrm{eff}}} \sum_{\ell=0}^\infty \frac{1}{\ell!}\langle \tau_{a_1}(\gamma_1),\dots, \tau_{a_n}(\gamma_n),\underbrace{\tau_0(\btau),\dots,\tau_0(\btau)}_{\text{$\ell$ times}}\rangle_{g,n+\ell,d}^{\cX}.
 \end{align}
\end{definition}
We do not use Novikov variables since the convergence issue regarding $\btau$ is resolved in \cites{Iritani07,CoIr15}, or after invoking the mirror theorem and the oscillatory integral expression of the $I$-function. The Equation \eqref{eqn:correlator} is a complex analytic function of $e^{\btau}$ in $\Uee$ for small $\epsilon$. We further define
\[
\llangle f(\psi/z),\dots \rrangle_{g,h}^X=\sum_{i\geq 0} z^{-i}\llangle a_i \psi^i,\dots\rrangle_{g,h}^X,
\]
where $f=\sum_{i\geq 0} a_i z^i$ is an analytic function at $0$.

Define a degree operator $\widetilde{\deg}: H^*(X)\to H^*(X)$ such that  \[\widetilde{\mathrm{deg}}\vert_{H^{2p}(X)}=(p-\frac{\dim(X)}{2})\mathrm{id}_{H^{2p}(X)}.\]
We cite the main theorem from Iritani's paper \cite[Theorem 4.11]{Iritani09}.
\begin{theorem}
  \label{thm:iritani}
  There is an isomorphism $\si:K(X)\to \bH_q$ such that
  \begin{itemize}
    \item $\chi(V,W)=S(\si(V),\si(W)),\forall V,W\in K(X)$.
    \item We have the following identity
    $$\llangle1, \frac{z^{-\widetilde \deg}z^{c_1(X)}\boldA_V}{z+\psi}\rrangle_{0,2}^X=(-2\pi z)^{-n/2}\int_{\si(V)}e^{-\frac{W}{z}}\Omega.$$
  \end{itemize}
  Here $z>0$ and \emph{higher asymptotic classeses} $\boldA_V=\Ch(V)\hat\Gamma_X$, where $\Ch(V)$ is the modified Chern character
  \[
  \Ch(V)=\sum_{p\geq 0}(2\pi\bi)^p \ch_p(V).
  \]
  We understand $z^s$ as $\exp(s\log z)$ for $\log z\in \bR$. The \emph{Gamma class}
  \begin{equation}
    \label{eqn:Gamma}
    \hat\Gamma_X=\sum_{i=1}^k\Gamma(1+\delta_i),
  \end{equation}
  where $\delta_i$ are Chern roots of $TX$.
\end{theorem}

\subsection{Quantum cohomology and Frobenius algebra}

The dimension of $H^*(X;\bC)$ is $\fs$ -- recall that in Section \ref{sec:LG-B-model} we have chosen $e^\vee_1,\dots,e^\vee_\fp\in \bL^\vee\cong H^2(X;\bZ)$ -- we let $H_i=e^\vee_i$ for $i=1,\dots,\fp$ and let $H_1,\dots,H_\fs$ be a homogeneous basis of $H^*(X;\bC)$. It dual basis in $H^*(X;\bC)$ are $H^1,\dots, H^\fs$.

The \emph{small quantum product} of the toric Fano variety $X$ is defined as the following
\begin{align*}
 \alpha\star_\btau\beta&=\sum_{\ell\geq 0,\beta\in \bK_\eff}\sum_{a=1}^\fr \langle\alpha,\beta,H_a,\btau,\dots,\btau\rangle_{0,3+\ell,\beta}^X H^a\\
  &=\sum_{\beta\in \bK_\eff(X)} \sum_{a=1}^\fr \langle \alpha,\beta,H_a\rangle_{0,3,\beta}^X e^{\langle\btau,\beta\rangle} H^a.
\end{align*}
where $\btau=\tau_1 H_1+\dots +\tau_\fp H_\fp\in H^2(X)$, $e^{\langle\btau,\beta\rangle}=e^{\tau_1\langle\beta,H_1\rangle }\dots e^{\tau_n\langle\beta,H_\fp\rangle}$ while the second equality is from the divisor equation. Since $X$ is Fano, there is only finitely many curve classes $\beta$ such that $\langle \alpha,\beta, H_a\rangle_{0,3,\beta}^X\neq 0$ due to the dimension reason, so the above definition is well-defined for all $\btau$.\footnote{The convergence of the big quantum cohomology when $\btau$ is not necessarily in degree $2$ is unknown in general.}

The quantum product $\star_\btau$ and together with the usual non-degenerate pairing
\begin{align}
  \label{eqn:cohomology-pairing}
  (\alpha,\beta)=\int_X \alpha\cup\beta,\ \alpha,\beta\in H^*(X),
\end{align}
gives the cohomology group $H^*(X)$ a Frobenius algebra structure. We denote this algebra by $QH^*_\btau(X)$.

We say the quantum cohomology $QH^*(X)$ is \emph{semisimple} at $\btau$ if there are basis $\phi_1,\dots,\phi_\fs\in H^*(X)$ such that
\[
\phi_i\star_{\btau}\phi_j=\delta_{ij} \phi_i.
\]
This basis $\{\phi_i\}_{i=1}^\fs$ is called the \emph{idempotent basis} or \emph{canonical basis}. It is unique up to permutation. Being semisimple is an open condition for the parameter $\btau$, and for toric Fano variety $X$ the quantum cohomology is generically semisimple on $e^\btau \in \Uee$ for sufficiently small $\epsilon$ \cite[Corollary 4.9]{Iritani09}. The canonical basis is then an $H^*(X)$-valued function of $\btau$ and we denote it by $\phi_i(\btau)$ in case we want to emphasize its dependence on $\btau$.

We define the following quantum connection as a meromorphic connection on the trivial bundle $\cF: H^*(X)\times (H^2(X)\times \bP^1)\to H^2(X)\times \bP^1$
\begin{align}
  \nabla_{\alpha}&=\partial_\alpha+\frac{1}{z}\alpha\star_\btau,
  \label{eqn:qde-tau}
  \\
  \nabla^\btau_{z\partial z}&=z\frac{\partial}{\partial z}-\frac{1}{z}(E\star_\btau)+\widetilde{\mathrm{deg}}.\label{eqn:qde-tau-off}
\end{align}
Here the \emph{Euler vector field} is
\[E=c_1(X)+\sum_{a=1}^\fr\left( 1-\frac{1}{2}\deg H_a\right)\tau_a H_a.\]
Note that $\nabla_{z\partial z}^\btau$ does not  involve taking derivative in the $\btau$-direction.

The \emph{mirror map} $q=q(\btau)$ relates B-model parameters $q=q_1,\dots,q_\fp$ with A-model parameters $\tau_1,\dots, \tau_\fp$\footnote{Since $X$ is Fano the mirror map takes such a simple form.}
\begin{equation}
q_a=\exp(\tau_a),\ a=1,\dots,\fp.
\label{eqn:mirror-map}
\end{equation}
We denote $\tUe=q^{-1}(\Ue)$, $\tUec=q^{-1}(\Uec)$, $\tUee=q^{-1}(\Uee)$, and $\tUeec=q^{-1}(\Ueec)$.

\subsection{Flat sections from Gromov-Witten potentials}

We define the following descendants operator for $\alpha\in H^*(X)$.
\[
L(\btau,z)\alpha:=e^{-\frac{\btau}{z}}\alpha -\sum_{d\in \NE_X}\sum_{a=1}^\fs H^a\langle H_a,\frac{e^{-\frac{\btau}{z}}\alpha}{z+\psi}\rangle_{0,2,d}e^{\langle\btau,d\rangle}.
\]
The convergence of this operator $L(\btau,z)$ is known on $\tUe \times \bC^*$ for sufficiently small $\epsilon$ since the convergence of $\star_\btau$ and the fact that $L(\btau,z)\alpha$ satisfies Equation \eqref{eqn:qde-tau} \cite[Proposition 2.4]{Iritani09}.

\begin{proposition}[Definition 2.5, Equation (19) and Definition 2.9 of \cites{Iritani09}]
  For any $K$-group element $E\in K(X)$, the generating function
  \begin{equation}
    \label{eqn:S}
    \cZ(E)=L(\btau,z)z^{-\widetilde{\deg}}z^{c_1(X)} \Ch(E)\hat\Gamma_X=\sum_{i=1}^\fs \llangle H_i,\frac{z^{-\widetilde{\deg}}z^{c_1(X)}\boldA_{[E]}}{z+\psi}\rrangle_{0,2}^X H^i.
  \end{equation}
  is a flat section of the quantum connection $\nabla$.
\end{proposition}
\begin{remark}
  Setting $z>0$ and choose $\log z\in \bR$ as in Theorem \ref{thm:iritani}. We get a single valued section on $\Uee\times \{0,\infty\}$. When we fix $\btau\in \Uee$ as a constant, this provides a solution to \eqref{eqn:qde-tau-off}.
\end{remark}

\subsection{Mirror symmetry for quantum cohomology}

Here we introduce a related B-model description of the D-module in \cite{Iritani09}. We do not explicitly describe an B-model D-module but just state facts from \cites{Iritani09} for our purposes here.

Let $\cM=(\bC^*)^\fp$. We exclude bad points and denote the remaining Zariski open and dense $\cM^\circ$ such that for $q\in \cM^\circ$, $W_q$ is \emph{non-degenrate at infinity} \cite[Definition 3.6]{Iritani09}. For sufficiently small $\epsilon$, $\Ue\subset \cM^\circ$ \cite[Lemma 3.8]{Iritani09}.

Following \cite[p1043]{Iritani09}, we define
\[
R^\vee_{\bZ,(q,z)}=H_n(Y_q,\{\mathrm{Re}(W_q/z)\gg 0\};\bZ),
\]
where $(q,z)\in \cM^\circ\times \bC^*$. These relative homology groups form a local system $R^\vee_\bZ$ of rank $\fr$ over $\cM^\circ\times \bC^*$. When $z>0$, $R^\vee_{\bZ,(q,z)}=\bH_q$ in our notion. For $q\in \Uee$ with small $\epsilon'$, any $\Gamma \in \bH_q$ extends to a flat section of $R^\vee_\bZ$ over $\Uee\times \widetilde \bC^*$ where $\widetilde \bC^*=\bC$ is the unviersal cover of $\bC^*$.

Let $\sf_a=\frac{\partial W}{\partial \tau_a}\in \cO_{\cM^\circ \times \bC^*}[X_1^\pm,\dots,X_n^\pm]$ for $a=1,\dots,\fp$.  We define an operator $D_a=-z\frac{\partial}{\partial\tau_a}+\sf_a$. Since $\{H_a\}_{a=1}^\fs$ multiplicatively generate $QH_\btau^*(X)$, one may write $H_i,\ i=\fp+1,\dots,\fr$ in the following form
\begin{align*}
  H_i&=\sum_{j=1}^{k_i} A_{ij}(e^\btau)H_{s_{ij_1}}\star_\btau \cdots \star_\btau H_{s_{ij_{r_{ij}}}},\quad s_{ij_l}\in\{1,\dots,\fs\},\\
  H^i&=\sum_{j=1}^{k^i} A^i_j(e^\btau)H_{s^i_{j_1}}\star_\btau \cdots \star_\btau H_{s^i_{j_{r^i_j}}},\quad s^i_{j_l}\in\{1,\dots,\fs\}.
\end{align*}
Then for $i=\fp+1,\dots,\fs$ we define
\begin{align*}
&\sf_i=\sum_{j=1}^{k_i}  A_{ij}(q)D_{s_{ij_1}}\cdots D_{s_{ijr_{ij}}}1,\\
&\sf^i=\sum_{j=1}^{k^i}  A^i_{j}(q)D_{s^i_{j_1}}\cdots D_{s^i_{jr^i_{j}}}1.
\end{align*}
By the definition of $\sf_i,\sf^i$ are in $\cO_{\cM^\circ \times \bC^*}[X_1^\pm,\dots,X_n^\pm]$ and actually they are a polynomial in $z$. We cite Iritani's identification of these two quantum D-modules \cite[Proposition 4.8]{Iritani09} in the following way.
\begin{theorem}[Iritani]
  \label{thm:D-mod-isomorphism}
  Over $q\in \Ue$ for some small $\epsilon>0$, there is an morphism
  \[
  \Mir: (\btau\times \mathrm{id})^* (\cF/H^2(X;\bZ)) \to \cO_{\cM^\circ\times \bC^*}
  \]
  such that
  \begin{itemize}
    \item The quotient of $H^2(X;\bZ)$ is on the base of this D-module $H^*(X)\times \bP^1$. Indeed the quantum cohomology at $\btau$ is the same as $\btau+2\pi\bi\btau'$ for any $\btau'\in H^2(X;\bZ)$.
    \item $\Mir(H_{a_1}\star_\btau \dots\star_\btau H_{a_s})=(D_{a_1}\dots D_{a_s} 1),$ and then $\Mir(H_i)=\sf_i, \Mir(H^i)=\sf^i$.
    \item The image of $\Mir$ is also equipped with a D-module structure. In particular, a flat section $\sf$ is characterized by the following pairing with any flat section $\Gamma$ of $R^\vee_{\bZ}$ being constant
    \[
    \int_\Gamma \sf e^{-\frac{W}{z}}\Omega.
    \]
    The map $\Mir$ preserves this flat structure.
  \end{itemize}
\end{theorem}

This theorem implies the following.
\begin{corollary}
  \label{cor:B-S-flat}
  For a flat section $\Gamma$ of $R^\vee_\bZ$,
\[
s_\Gamma:=(-2\pi z)^{-n/2}\sum_{i=1}^\fs\left( \int_\Gamma \sf_i e^{-\frac{W}{z}}\Omega \right) H^i
\]
is a flat section of $\cF$. Moreover, given $\Gamma \in \bH_q$ with $q\in \Uee$ with small $\epsilon'$, by extending $\Gamma$ to a multi-valued section of $\Uee\times \bC^*$,
\[
(-2\pi z)^{-n/2}\sum_{i=1}^\fs\left( \int_\Gamma \sf_i e^{-\frac{W}{z}}\Omega \right) H^i
\]
is a multi-valued flat section of $\cF$ over $q\in \Uee\times \bC^*$.

\end{corollary}

Following \cite[Section 3.2.3]{Iritani09}, we introduce the Jacobian ring as
\[
\Jac(W)=\frac{\bC[X_1^\pm,\dots,X_n^\pm]}{\langle \frac{\partial W}{\partial X_1},\dots,\frac{\partial W}{\partial X_n}\rangle},\quad \Jac(W_q)=\Jac(W)\otimes_{\bC[q^\pm]}\bC_q.
\]
We notice that $\Jac(W)$ is a $\bC[q^\pm]=\bC[q_1^\pm,\dots,q_\fp^\pm]$-algebra. For any $f\in \bC[X_1^\pm,\dots,X_n^\pm]$, we denote its class in $\Jac(W_q)$ to be $[f]$. The residue pairing on $\Jac(W_q)$ is given by
\begin{equation}
  \label{eqn:residue-pairing}
([f],[g]):=\frac{1}{(2\pi\bi)^n}\int_{|dW|=\delta} \frac{fg\Omega}{\prod_{i=1}^\fr X_i\frac{\partial W}{\partial X_i}}.
\end{equation}
The mirror theorem of Givental/Lian-Liu-Yau \cites{Gi96a,Gi96b,LLY97}, or the direct proof of Fukaya-Oh-Ohta-Ono states \cites{FOOO10} the following ring isomorphism under the mirror map \eqref{eqn:mirror-map}.
\[
\mir: QH^*_\btau(X)\xrightarrow{\sim} \Jac(W_{q(\btau)}).
\]
Moreover, the pairing $(,)$ of $\Jac(W_q)$ is indenfitified with the cohomology pairing \eqref{eqn:cohomology-pairing} with the residue pairing \eqref{eqn:residue-pairing}.
\begin{remark}
The isomorphism $\mir$ is the map $\Mir$ restricted to $z=0$
\[
\Mir\vert_{z=0}: QH^*_\btau(X)\xrightarrow{\sim} \Jac(W_{q(\btau)})
\]
after one regards the elements of the Jacobian ring in $\cO_{\cM^\circ\times \bC^*}[X^\pm]$.
\end{remark}

The Jacobian ring $\Jac(W_q)$ is semisimple if and only if it is holomorphic Morse. We denote the critical points of $W_q$ to be $p_1,\dots,p_\fr$ such that $W(p_i)=\mathsf{cr}_i$. Recall that $\Ueec$ is the dense and open subset of $\Uee$ where $W_q$ is holomorphic Morse. We list some properties of $\Jac(W_q)$ and $\mir$ which are derived directly from the definition.
\begin{proposition}
  \label{prop:identification-mir}
  When $W_q$ is holomorphic Morse we have the following
\begin{itemize}
  \item $[f]=[g]\in \Jac(W_q)\Leftrightarrow f(p_i)=g(p_i),\ i=1,\dots,\fs.$
  \item The canonical basis of $\Jac(W_q)$ is $[\varphi_i]$ such that $\varphi_i(p_j)=\delta_{ij}$.
  \item The length of $[\varphi_i]$ is $1/\sqrt{\det(\Hess_{p_i}(W_q))}$.
  \item The map $\mir$ identifies
  \begin{equation}
    \label{eqn:QH=Jac-1}
    \mir(H_a)=[\frac{\partial W_q}{\partial \tau_a}],\quad \mir(D_i)=[\tX_i],\quad \mir(c_1(X))=[W_q].
  \end{equation}
  where $a=1,\dots,\fp$, $i=1,\dots,\fr$.
  \item The map $\mir$ relates canonical basis
  \begin{equation}
    \label{eqn:QH=Jac-2}
    \mir(\phi_i)=[\varphi_i]\in H^*(X),\quad \Delta_i:=1/(\phi_i,\phi_i)=-\det(\Hess_{p_i}(W_q)).
  \end{equation}
\end{itemize}
\end{proposition}

\section{HMS with sheaves: Coherent-constructible correspondence}

\label{sec:HMS-sheaf}

\subsection{Categorical notions}

In this spaper, we work in the setting of dg or $A_\infty$, although essentially most of our arguments are only needed on the level of K-theory. We also omit ``quasi-'' when talking about equivalences of categories.

We use $\Coh(X)$ to denote its dg category of coherent sheaves on $X$ (a dg enhancement of the usual derived category) which is smooth and proper. We follow the notation of \cite{Nadler16}.  For a real analytic manifold $B$, let $\Sh^\diamond(B)$ be the big dg category of $\bC$-modules (possibly with unbounded cohomologies sheaf), and let us define the dg cat $\Sh^\diamond_\Lambda(B)$ as the full subcategory spanned by objects with singular support in $\Lambda \subset T^*B$. Let $\Sh^w_\Lambda(B)$ be the full subcategory of compact objects in $\Sh^\diamond_\Lambda(B)$, called {\em wrapped microlocal sheaves} \cite[Definition 1.3]{Nadler16}, and let $\Sh_\Lambda(B)$ denote the traditional constructible sheaf with bounded and constructible cohomology sheaf. \footnote{In the literature, the wrapped microlocal sheaf is sometimes denoted as $\Sh_\La(B)^\mathrm{c}$, where $\mathrm{c}$ denote compact object as in \cite{GPS18-2}. Note however, the traditional constructible sheaf is also denoted as $\Sh^\mathrm{c}$, as in \cite{Kuw16}. To avoid confusion, we use the original notation of Nadler $\Sh^w$ for wrapped sheaves.}

\begin{remark}
In the case of smooth projective toric manifold (even smooth DM toric stack), the FLTZ skeleton $\Lambda$ ensures that
\[ \Sh^w_\La(T^n) \cong \Sh_\La(T^n). \]
\end{remark}

As a variant of this notion when $\Lambda^\infty \In T^\infty B \cong S^*B$, $\Sh_{\linf}(B)$ denotes the full subcategory whose singular support's infinity is in $\linf$.

\color{black}

For any object $F$ in a triangulated (or $A_\infty$, dg) category $\cC$, we use $[F]_K$ to denote its $K$-group element.

\subsection{Coherent-constructible correspondence (CCC)}
Let $X$ be the complete simplicial toric variety defined by the $\Si$. We define a conical Lagrangian in $T^*M_\bR=M_\bR\times N_\bR$
\[
\tilde\Lambda=\bigcup_{\si\in\Si} (M+\si^\perp) \times (-\si)
\]
where
\[
\si^\perp=\{u\in M_\bR\vert \langle u, v\rangle =0,\forall v\in \si \}.
\]

Coherent-constructible correspondence, observed by \cites{Bondal06}, relates coherent sheaves to constructibles of certain polyhedral types. For any equivariant anti-ample line bundle $E\in \Coh_\bT(X)=\sum_{i=1}^\fr\cO(c_i D_i)$, the set characterized by
\[
\Delta_E=\{x\in M_\bR, \langle b_i, x\rangle < c_i\}
\]
is an open polytope in $M_\bR$ such that each vertex corresponds to a top dimensional cone in $\Si$. We call polytopes of such form \emph{toric polytopes}.  The equivariant coherent constructible correspondence \cites{FLTZ11} says the following.
\begin{theorem}[Fang-Liu-Treumann-Zaslow \cites{FLTZ11}]
  \label{thm:equivariant-ccc}
  There is an equivalence
  \[
  \iota: \Coh_\bT(X)\to \ShtL,
  \]
  which sends $E=\cO(\sum_{i=1}^\fr c_i D_i)$ to $i_{*}\bC_{\Delta_E}$ where $i: \Delta_E\hookrightarrow M_\bR$ is the embedding, and $\bC_{U}$ is the constant sheaf $\bC_U$ on an open set $U$.
\end{theorem}

This theorem has a non-equivariant version as follows. Since $\tilde \Lambda$ is invariant under the translation of $M$, we denote $\Lambda=\tilde\Lambda/M\subset T^*M_T=N_\bR\times (M_\bR/M)$ and $M_{T,\si}=\si^\perp/M\subset M_T$. Define $\Lambda_\si=M_{T,\si}\times (-\si)\subset \Lambda$.

\begin{theorem}[\cites{Kuw16, Vaintrob16, ZhouPeng17b}]
  \label{thm:ccc}
  There is an equivalence
  \[
  \iota: \Coh(X)\to \ShL.
  \]
\end{theorem}
\begin{remark}
  The functor $\iota$ in Theorem \ref{thm:equivariant-ccc} passes to a fully faithful exact functor \cites{Treumann10} in the non-equivariant setting -- then one needs to show that this is an equivalence. By a slight abuse of notation we use $\iota$ to denote both equivariant and non-equivariant CCC.
\end{remark}
We denote $\mathcal S$ as a (Whitney) stratification on $M_T$ such that each linear Lagrangian in $\Lambda$ is contained in a $T^*_N S$ for $S\in \mathcal S$. Let $\tilde{\mathcal S}$ be the lift of $\mathcal S$ on $M_\bR$.

\subsection{Oscillatory integrals on characteristic cycles}

We prove the following proposition, which allows us to do integration on characteristic cycles after identifying $M_\bR$ with $N_\bR$ by $\Psi$ as in Section \ref{sec:lg-a}.

\bpp
\label{prop:La-integrable}
For each small $\epsilon'' > 0$ and $\epsilon>0$, we can find a sufficiently small $\epsilon'$, such that for $q\in \Uee$ there exists a conical neighborhood $V \In T^*M_T \RM M_T$ containing $\Lambda_\Sigma \RM M_T$ and a compact set $K \in T^*M_T$, and the image of $W_q(\Psi^{-1}(V \RM K))$ lies in the sector $\{ | \arg z | < \epsilon'' \} \in \C$.
\epp
\bpf

For each positive dimensional cone $\sigma \In \Sigma$, we consider the corresponding component of $\Lambda_{\sigma}$ and the corresponding rays in $\sigma$. Recall that monomials in $W_q$ are labelled by $\Sigma(1)$. Let $W_{q,\sigma}$ be the partial sum consisting of terms in $\sigma\cap \Si(1)$. We claim the other terms are small in the following sense. Fix a norm $\| - \|$ on $N_\R$.

{\bf Claim:} There exists $\delta>0, C>0$, such that for all $(\theta, p) \in \La_\sigma$,
\[  \frac{| W_q - W_{q,\sigma}| }{| W_{q,\sigma}|} \big|_{\Psi^{-1} (\theta, p)} < C e^{ - \delta \| p \| }. \]
Furthermore, there is an open neighborhood $V_\sigma$ of $\La_\sigma$, such that the above estimate holds and $|\arg(W_{q,\sigma})| < \epsilon''/2$.

Assuming this claim, it is easy to check that, if we take $$V = \bigcup_{0 \neq \sigma \In \Sigma} V_\sigma,$$ then the condition is satisfied by taking $K = \{ \| p \| < R\}$ for large enough $R$.

Now we prove the claim. It is easy to check that if $q$ is real (meaning all $q_1,\dots,q_\fp$ are real) if $(\theta, p) \in \La_{\sigma} \In T^*M_T$, then $W_\sigma(\Psi^{-1}((\theta, p))) > 0$. Recall that $A$ are vertices of the Newton polytope $Q$ of $W$, which is also the set of primitive vectors for rays in $\Sigma$. Let $A_\sigma = \sigma \cap A$, and $F_\sigma = \pa Q \cap {\sigma}$ be the closed face. Let $\psi$ be the Legendre transformation of $\varphi$, then in as discussed in \cites{ZhouPeng18} (Corrolary 2.14)  $\psi$ is ``adapted to $Q$'', i.e. each face of $Q$ has a minimum of $\psi$ in its interior. We define
\[ f_\sigma(p) := \max_{\alpha  \in A \RM A_\sigma}  \la \alpha, \| p \|^{-1} d \psi(p) \ra - \max_{\beta \in A_\sigma}  \la  \beta, \| p \|^{-1} d \psi(p) \ra. \]
Then since $\psi$ is ``adapted to $Q$'',
\[ f_\sigma(p) > 0, \quad \forall p \in F_\sigma. \]
Since $F_\sigma$ is compact, we have
\[ \delta_\sigma = \min_{p \in F_\sigma} f_\sigma(p) > 0. \]
Thus
\[  \frac{| W - W_\sigma| }{| W_\sigma|} \big|_{\Psi^{-1} (\theta, p)} \leq B|A| e^{- \delta_\sigma \|p\|} \]
for $(\theta, p) \in \La_\sigma = M_{T, \sigma} \times (- \sigma)$, where $B$ is the largest of the ratio among $|c_\alpha(q)|$, which has a upper bound for $q\in \Uee$ with a fixed $\epsilon$.
There is neighborhood $U_\sigma$ of $F_\sigma$ in $\pa Q$, such that
\[  \min_{p \in U_\sigma} f_\sigma(p) > \delta_\sigma/2 \]
Thus
\[  \frac{| W - W_\sigma| }{| W_\sigma|} \big|_{\Psi^{-1} (\theta, p)} \leq |A| e^{- \frac{\delta_\sigma}{2} \|p\|} \]
for $(\theta, p) \in M_{T, \sigma} \times (\R_{<0} U_\sigma)\subset T^*M_T$.
Finally, Let $k = \dim \sigma$, then $W_\sigma$ has $k$ terms. On $M_{T, \sigma}\times M_\bR$ these $k$-terms' arguments are controlled by $\epsilon'$ -- we choose sufficiently small $\epsilon'$ that these arguments are less than $\epsilon''/(4k)$. Let $\wt M_{T, \sigma}$ be a neighborhood of $M_{T, \sigma}$ on $M_T$ where these $k$-terms has arguments in $(-\epsilon''/(2k), \epsilon''/(2k))$. Then, we may verify that $\wt M_{T, \sigma} \times (\bR_{<0} U_\sigma)$ is the desired neighborhood $V_\sigma$ of the claim. We thus finishes the proof of the claim, and of the proposition.
\epf

\begin{corollary}
Let $F\in \ShLc$, then the following integration when $|\arg z|<\frac{\pi}{2}$
\[
\int_{\CC(F)}e^{-\frac{W}{z}}\Omega:=\int_{\Psi^{-1}(\CC(F))} -e^\frac{W}{z}{\Omega}.
\]
is well-defined for on some $\Uee$.
\end{corollary}
\begin{proof}
  The characteristic cycles are supported in the singular support $\Lambda$. We choose $\epsilon,\epsilon'$ for $\epsilon''=\frac{\pi}{2}-|\arg z|$ by Proposition \ref{prop:La-integrable}.
\end{proof}


\subsection{Iritani's isomorphism and oscillatory integral}

Recall that Iritani's result Theorem \ref{thm:iritani} \cite[Theorem 4.11]{Iritani09} gives an identification of $\si: K(X)\xrightarrow{\sim} \bH_q$. Let $F\in \ShLc$. The characteristic cycles of $F$ is a Lagrangian cycle in $T^*M_T$, and under the identification $\Psi$ it represents a class in $\bH_q$ by Proposition \ref{prop:La-integrable}. We denote this class by $[\Psi^{-1}(\CC(F))]\in \bH_q$. 

As discussed in Corollary \ref{thm:D-mod-isomorphism},
\[
s_\Gamma=(-2\pi z)^{-n/2}\sum_{i=1}^\fs\left ( \int_\Gamma \sf_i e^{-\frac{W}{z}}\Omega\right)  H^i
\]
is a flat section of $\cF$.

The main theorem of \cite{Fang16} tells us for $z>0$ and $q\in \Uee$ with small $\epsilon$ and $\epsilon'$
\[
  \llangle1, \frac{z^{-\widetilde \deg}z^{c_1(X)}\boldA_{[E]}}{z+\psi}\rrangle_{0,2}^X=(-2\pi z)^{-n/2}\int_{\CC(\iota(E))}e^{-\frac{W}{z}}\Omega
\]
for any $E\in \Coh(X)$. The fact following fact is a direct application of divisor equation for Gromov-Witten invariants.
\begin{equation}
  \label{eqn:take-derivative}
  z\frac{\partial}{\partial \tau_a}\llangle \alpha, \frac{\beta}{z+\psi}\rrangle_{0,2}^X=\llangle H_a\star_\btau\alpha,\frac{\beta}{z+\psi}\rrangle_{0,2}^X.
\end{equation}

Then we have the following proposition.
\begin{proposition}
  \label{prop:any-primary-insertion}
  For $z>0$ and $q\in \Uee$ with small $\epsilon$ and $\epsilon'$,
  \[
  \llangle H_i,\frac{z^{-\widetilde \deg}z^{c_1(X)}\boldA_{[E]}}{z+\psi}\rrangle_{0,2}^X=(-2\pi z)^{-n/2}\int_{\CC(\iota(E))} e^{-\frac{W}{z}}\sf_i\Omega.
  \]
\end{proposition}
\begin{proof}
    For $s_i\in\{1,\dots,\fp\}$ and Equation  \eqref{eqn:take-derivative}
  \begin{align*}
    &\llangle H_{s_r}\star_\btau\dots \star_\btau H_{s_1},\frac{z^{-\widetilde \deg}z^{c_1(X)}\boldA_{[E]}}{z+\psi}\rrangle_{0,2}^X=(-2 \pi z)^{-n/2}\int_{\CC(\iota(E))} e^{-\frac{W}{z}}(D_{s_r}\dots D_{s_1}f)\Omega,
  \end{align*}
  By the definition of $\sf_i$,
  \begin{align*}
    & \llangle H_i,\frac{z^{-\widetilde \deg}z^{c_1(X)}\boldA_{[E]}}{z+\psi}\rrangle_{0,2}^X=(-2\pi z)^{-n/2}\int_{\CC(\iota (E))} e^{-\frac{W}{z}}\sf_i \Omega,\quad i=1,\dots,\fr.
  \end{align*}
\end{proof}
So as an flat section of $\cF\vert_{\Uee\times \{z>0\}}$, $\cZ(E)=s_{\Psi^{-1}(\CC(\iota(E)))}$ for any $E\in \Coh(X)$, or one can extend to sections over $\Uee\times \bC^*$ and they are equal as multi-valued sections. In particular this matches Iritani's isomorphism $\si$ and the $K$-theoretic level of coherent-constructible correspondence functor $\iota$.
\begin{theorem}
  \label{thm:central}
  \[\si([E]_K)=[\Psi^{-1}(\CC(\iota(E)))].\]
\end{theorem}

\section{HMS with Lagrangians}
\label{sec:HMS-Lag}

\subsection{Wrapped Fukaya categories $\cWL$ and $\cWFS$.}
We consider the following partially wrapped Fukaya categories, defined on a Liouville sector with stop \cites{GPS18-1}:
\begin{itemize}
  \item The partially wrapped Fukaya category $\cWL$ for $T^*M_T$ with stop at $\Lambda^\infty$. We refer to the definition of partially wrapped Fukaya in \cites{GPS17,GPS18-1}. Notice we have a Liouville manifold $T^*M_T$ with stop $\Lambda^\infty$, which is equivalent to a \emph{Liouville sector} in \cites{GPS17}. We only remark that the admissible Lagrangian as objects are cylindrical (conical) outside a compact set and do not intersect with $\Lambda^\infty$ at infinity.
  \item The Fukaya-Seidel category $\cWFS$. We simply adopt the definition of Ganatra-Pardon-Shende \cite[last example on p2]{GPS17}, where the superpotential is $W:M_{\bC^*}\to \bC$ and we identify $M_{\bC^*}$ with $T^* M_T$ by $\Psi$. We consider the partially wrapped Fukaya category $\cWFS:=\cW(T^*M_T; W^{-1}(+\infty))$ which is defined as follows. The superpotential $W$ gives an embedding of the total space of an $F_0$-bundle over $S^1$ to $\partial_\infty (T^*M_T)$ with the contact form $dt-\lambda$, where the Liouville domain $F_0$'s completion is the generic smooth fiber $F$ of $W$ at a very large positive real number, $\lambda$ is the Liouville form on $F_0$, and $t\in S^1$. The circle $S^1$ captures the argument of the superpotential $W$ at infinity. The category $\cW(T^*M_T; W^{-1}(+\infty))$ is simply defined to be the partially wrapped Fukaya category for $T^*M_T$ with stop $\{1\}\times F_0$.

   It is shown in \cite[Corollary 2.9]{GPS18-1} that
  \[
  \cWFS=\cW(T^*M_T;W^{-1}(+\infty))\cong \cW(T^*M_T,\mathfrak c_{F}).
  \]
  Here $\mathfrak c_F$ is the \emph{core} (the part that under Liouville flow does not go to infinity) of the generic fiber $F$ or $F_0$. In particular when $q$ is real, by \cites{ZhouPeng18,GammageShende17} we know that $\Lambda^\infty$ is the core of $F_0$ and then
  \[
  \cWFS\cong \cWL.
  \]
  Since $\cWFS$ does not depend on $q$ for its small perturbation, we know $\cWFS\cong \cWL$ for $q\in \Ueec$ for small $\epsilon$ and $\epsilon'$.
\end{itemize}

\subsection{$\cWL$ and the microlocalization functor}

We use \cites{GPS18-2}'s result to describe the equivalence between the partially wrapped Fukaya category of a cotangent bundle and the category of constructible sheaves on its base manifold. An earlier result of \cites{Nadler09, NaZa09} equates the infinitesimally wrapped Fukaya category and constructible sheaves. We use partially wrapped categories since the objects we consider (thimbles) fit more suitably. \footnote{It would be interesting to see if thimbles are in the infinitesimally wrapped Fukaya-Nadler-Zaslow category.} We cite the main theorem from \cites{GPS18-2}.
\begin{theorem}[Ganatra-Pardon-Shende]
  \label{thm:GPS}
  There is a microlocalization functor
  \[
    \mu: \Sh_\La^w(M_T) \xrightarrow{\sim} \Perf\cWL,
  \]
  which is a quasi-equivalence of $A_\infty$-categories.
\end{theorem}
Note that since we use the Liouville one-form $\lambda = -p dx$ on the cotangent bundle, we do not have $\Perf\cWLo$ but rather $\Perf\cWL$ in the above equivalence.
By a slightly abuse of notation we also use $\mu$ to denote the quasi-equivalence between $\ShtLc$ and $\cWtL$.

Let $U$ be an (open) toric polytope in $M_\bR$, and $\tF$ be the associated costandard sheaf in $\ShLc$ while $F$ be the associated costandard sheaf in $\ShLc$.
Let $\tL'$ be the Lagrangian brane graph $\Gamma_{-d\log f}$, where $f: \bar U\to \bR$ is a smooth function that
\[
f|_{U}>0,\ f(\partial U)=0;
\]
and $\Gamma_{d\log f}$ is the Lagrangian graph in $T^*M_\bR$. Define $L=\pi(L')$, where $\pi$ is the universal cover $T^*M_\bR \to T^*B$. After a finite but small positive push of $L'$ (resp. $\tL'$) at the infinity, we obtain $L\in \cWL$ (resp. $\tL \in \cWtL$). We call this Lagrangian $L$ and $\tL$ the costandard Lagrangian associated to $\cU\subset M_\bR$.\footnote{The terminology of a costandard Lagrangian is from \cites{NaZa09}, and the definition is slightly different since in the Nadler-Zaslow's infinitesimal Fukaya category $L'$ is the standard Lagrangian.}
\begin{proposition}
\[
\CC(\tF)=\CC(\mu^{-1}(\tL)).
\]
\label{prop:cc-standard}
\end{proposition}
\begin{proof}
  Let $\tF'=\mu^{-1}(\tL)\in \ShtLc$. Then
  \[
  \CC(\tF)=\sum_{i=1}^k c_i \Lambda_i,\quad
  \CC(\tF')=\sum_{i=1}^k c'_i \Lambda'_i,
  \]
  where $\Lambda_i$ (resp. $\Lambda'_i$) is a linear (not necessarily complete) Lagrangian $\Delta_i \times (-\si_i)$ (resp. $\Delta'_i\times (-\si'_i)$), where $\Delta_i$ and $\Delta'_i$ are strata in $\tilde{\mathcal S}$, while $\si_i$ (or $\si'_i$) are fans in $\Si$. In particular if $\si_i=\{0\}$ then $\Lambda_i=\Delta_i$ is an open set in $M_\bR$. Notice that $\CC(\tF)$ and $\CC(\tF')$ are determined by the coefficients of $c_i$ for $\si_i\neq \{0\}$ since $M_\bR$ is contractible.

  Pick an interior point $(x,p)$ of any $\Lambda_i^\infty$ where $x\in M_\bR$ and $p\in S^\infty_x M_\bR$. Then
  \[
  \chi_{\ShtLc}(\bC_p, \tF)=c_i=\begin{cases}
  1,\text{ if $x\in \partial U$},\\
  0,\text{ if $x\notin \partial U$}.
  \end{cases}
  \]
  while
  \[
  \chi_{\ShtLc}(\bC_p, \tF')=\chi_{\cWtL}(L_p,\tL)=c_i=\begin{cases}
  1,\text{ if $x\in \partial U$},\\
  0,\text{ if $x\notin \partial U$}.
  \end{cases}
  \]
  Here $L_p$ is the Legendrian linking disk at $p$ which intersects $\Lambda_i$ transversally at a single point.
\end{proof}

If a (not necessarily closed) submanifold $V$ in $M_\bC^*$ represents a class in $\mathbb H_n$ we denote such class by $[V]$. For example, $\Gamma_i$ is a thimble over a path pointing rightwards, like in Figure \ref{fig:topView}, then it represents a class $[\Gamma_i]\in \mathbb H_n$. By Proposition \ref{prop:La-integrable} a costandard Lagrangian $L$ also represents a class $[\Psi^{-1}(L)]\in \mathbb H_n$. We sometimes just write $[L]$ for $[\Psi^{-1}(L)]$ for $L\subset T^*M_T$.

Passing from $M_\bR$ to $M_T$, from Proposition \ref{prop:cc-standard} and the fact that $[\Psi^{-1}(L)]=[\Psi^{-1}(\CC(F))]$ in $\bH_q$ we know that
\begin{corollary}
For any costandard and standard Lagrangian $L\in \cWL$,
\[
[\Psi^{-1}(\CC(\mu^{-1} (L)))]=[\Psi^{-1}(L)].
\]
\label{cor:cycle-standard}
\end{corollary}

\subsection{Thimbles as objects in $\cWL$.}

\label{sec:thimbles-obj}
We require that $q\in \Ueec$. Then the superpotential $W_q$ has distinct Morse critical points with distinct critical values. Let $\theta$ be an admissible phase and $\Gamma_1,\dots,\Gamma_\fs$ be corresponding thimbles.

We cite \cite[Corollary 1.14]{GPS18-1}. As in \cites{GPS17, GPS18-1}, $\Gamma_i$ can be made into Lagragnian brane object $L_i$ in $\cWL$. Notice that when we say $[L]$ we always mean the class represented by a particular form of $L$ -- in our paper they are always (co)standard Lagrangians or thimbles.
\begin{proposition}
  \label{prop:exceptional}
  Thimbles $L_1,\dots,L_\fs$ form an exceptional collection, and they do generate $\cWL$.
\end{proposition}

\begin{lemma}
  \[
  \chi_\cWL([L]_K,[L']_K)=S([L],[L']),
  \]
  where $L$ and $L'$ is either a thimble object or a standard object in $\cWL$.
\end{lemma}
\bpf
Since $\cWL$ are generated by thimbles, without loss of generality we may assume $L, L'$ are thimbles, in particular lying over vanishing paths $\gamma, \gamma'$ in $\C$ tending to $+\infty$. By definition,  $\Hom_{\cWL}(L, L')$ is generated by the intersection points of thimbles $\wt L$ with $L'$, where $\wt L$ is the thimble over the counterclockwisely perturbed vanishing path $\wt \gamma$ of $\gamma$, such that $\wt \gamma \cap \gamma$ transversely. Hence both sides boils down to counting intersection points of $\wt L$ with $L'$ with signs, and the equality can be verified.
\epf

A thimble object $L_i$ corresponds to a class in $\bH_q$, denoted by $[L_i]$. By our construction $[L_i]=[\Gamma_i]$.

\begin{proposition}
  Assuming
  \[
  [L_i]_K=\sum_{j=1}^k c_{ij} [G_j]_K,
  \]
  where $G_j$ are standard Lagrangians, we have
  \[
  [L_i]=\sum_{j=1}^k c_{ij} [G_j].
  \]
  \label{prop:thimble-in-standards}
\end{proposition}
\begin{proof}
Let $F'$ be any standard Lagrangian in $\cW$.
\[
\chi_\cWL([L_i]_K,[F']_K)=\chi_\cWL(\sum_j c_{ij}[G_j]_K, [F']_K)= \sum_j{c_{ij}}S([G_j],[F']).
\]
On the other hand
\[
\chi_\cWL([L_i]_K,[F']_K)=S([L_i],[F']).
\]
Since $[F']$ could be chosen as any standard branes which span $\bH$, the fact that $S$ is a perfect pairing implies
\[
[L_i]=\sum_{j=1}^k c_{ij}[G_j].
\]
\end{proof}
Since standard Lagrangians generate $\cWL$, by Corollary \ref{cor:cycle-standard} and Proposition \ref{prop:thimble-in-standards} we have the following.
\begin{corollary}
  If $L$ is a thimble object in $\cW$, then
  \[
  [L]=[\Psi^{-1}\CC(\mu^{-1}(L))].
  \]
  \label{cor:cycle-thimble}
\end{corollary}

\section{Gamma II conjecture}

\label{sec:gamma-ii}

In this section we prove the Gamma II conjecture for complete smooth toric Fano variety $X$.

\subsection{Asymptotic flat sections and the Gamma II conjecture}

Recall that the quantum cohomology $QH_\btau^*(X)$ is semisimple for $\btau\in \tUee$ for sufficiently small $\epsilon$, with canonical basis $\{\phi_i\}$.

The quantum connection admits a set of asymptotic fundamental solutions. Let $\hat\phi_i=\phi_i/\sqrt{(\phi_i,\phi_i)}$ be the normalized idempotent basis. We define the map ${\bPsi}:\bC^\fs\to H^*(X)$ to be
\[
\bPsi\begin{pmatrix}a_1\\\vdots\\a_\fs\end{pmatrix}=a_1 \hat\phi_s+\dots a_\fs \hat\phi_\fs.
\]
Recall that $\{H^i\}_{i=1}^\fs$ is a basis of $H^*(X)$, and one writes $\bPsi$ as a matrix left multiplication
\begin{equation}
\label{eqn:Psi}
\bPsi=\begin{pmatrix} H^1,\dots,H^\fs \end{pmatrix}\Psi\cdot
\end{equation}
where $\Psi$'s $j$-th column vector is $\hat\phi_j$'s coodinates in $H^1,\dots,H^\fp$, i.e. the $(i,j)$-th element is $(H_i,\hat\phi_j)$.

The quantum multiplication $c_1(X)\star_{\btau}$ acts on $H^*(X)$ and its eigenvectors are $\phi_1,\dots,\phi_\fs$ with eigenvalues $u_1,\dots,u_\fs$, i.e. $c_1(X)\star_{\btau} \phi_i=u_i \phi_i$. Let $U$ be the diagonal matrix $\mathrm{diag}(u_1,\dots,u_\fs)$. Then the following well-known theorem provides asymptotic fundamental solutions.

\begin{theorem}[\cites{Du93,Gi01a}]
  \label{thm:Dubrovin-Givental-decomposition}
  When the small quantum cohomology $QH^*_{\btau}(X)$ is semisimple, the quantum connection \eqref{eqn:qde-tau-off} has the following fundamental solutions
\begin{equation}
  \label{eqn:Dubrovin-Givental-decomposition}
  \bPsi R(z) e^{-\frac{U}{z}},
\end{equation}
  where $R(z)=\mathrm{id}+R_1z+R_2z^2+\dots \in \mathrm{End}(\bC^\fs)\llbracket z\rrbracket$ is a matrix-valued formal power series in $z$. This formal solution is unique up to a signed permutation matrix multiplied from the right, which corresponds to the ambiguity of the order of $\Psi_a$.
\end{theorem}

Similarly to Section \ref{sec:thimbles}, we say a phase $\theta\in \bR$ is \emph{admissible} if $\mathrm{Im}(u_i e^{-\bi\theta})\neq \mathrm{Im}(u_j e^{-\bi\theta})$ for $u_i\neq u_j$, i.e. the segment between $u_i$ and $u_j$ is not  parallel to $e^{\bi\theta}$. We do not require $u_1,\dots, u_\fs$ are distinct but when $\btau\in \tUeec$, they are. These fundamental solutions have analytic lifts, as given in the following theorem.
\begin{theorem}\textup{\cite[Theorem 12.2]{Wasow1987}, \cite[Theorem A]{BJL1979}, \cite[Lectures 4,5]{Dubrovin1999}, \cite[Section 8]{BrTo2013}, \cite[Proposition 2.5.1]{GGI16}}
  At a semisimple point $\btau\in \Uee$, for an admissible phase $\theta\in \bR$, there exists $\delta >0$ and analytical fundamental solutions $Y_\theta(z)=(y_1^\theta(z),\dots, y_\fs^\theta(z))$ to the quantum connection \eqref{eqn:qde-tau-off} in the region $|\arg(z)-\theta|<\frac{\pi}{2}+\delta$ around $z=0$ such that
  \[
  Y_\theta(z) e^{\frac{U}{z}}\sim \bPsi R(z).
  \]
\end{theorem}

These analytic solutions $y_i^\theta$ can be analytically continuated, along a path in $\bC^*$, to $\arg(z)=0$, denoted by $\bar y_i^\theta$. The Gamma II conjecture is about how to express these solutions by $\mathcal Z([E])$ for some $E\in D^b\Coh(X)$.
\begin{conjecture}[Gamma II conjecture]
  \label{conj:gamma-ii}
  Assume the quantum cohomology of a Fano variety $X$ is semisimple at $\btau\in H^2(X;\bC)$ then any admissible phase $\theta$, there exists full exceptional collection $\{E^\theta_1,\dots,E^\theta_\fs\}$ in $D^b\Coh(X)$ such that $\bar y_i^\theta(z)=\cZ([E^\theta_i])$ for $i=1,\dots,\fs$.
\end{conjecture}

We will prove this conjecture when $X$ is a complete toric Fano smooth manifold near large radius limit point ($q\in \Ueec$) using inputs from enumerative and homological mirror symmetry.

\subsection{Proof of Gamma II conjecture}
\label{sec:proof}

We fix $q_0\in \Ueec$ and $\btau_0$ such that $q_0=q(\btau_0)$. Then $QH^*_{\btau_0}(X)$ is semisimple and $W_{q_0}$ is holomorphic Morse. Let $\theta$ be an admissible phase. For each critical value $\crv_i=W_{q_0}(p_i)$ one defines
the ray $\gamma_i=\crv_i+\bR_{\geq 0} e^{\bi\theta}$, i.e. $\gamma_i$ are rays starting from $\crv_i$ towards the direction of $e^{\bi\theta}$. Let $\Gamma^0_i$ be the Lefschetz thimbles associated to each $\gamma^0_i$, such that $W_{q_0}(\Gamma^0_i)=\gamma^0_i$. Then by the stationary phase expansion
\begin{align}
  \nonumber
  \int_{\Gamma^0_i} e^{-W_{q_0}/z} \sf_j \Omega &\sim (-2\pi z)^\frac{n}{2} \frac{e^{-W_{q_0}(p_i)/z}}{\sqrt{-\det\Hess_{p_i}(W_{q_0})}}(\sf_j(p_i)\vert_{z=0}+O(z))\\
  \label{eqn:asymptotic-expansion}
  &=(-2\pi z)^\frac{n}{2}e^{-u_i/z}(\hat\phi_i,H_j)(1+O(z)).
\end{align}
Here we use the fact that $\sf_j(p_i)\vert_{z=0}=(H_j,\hat \phi_i)$, $\hat \phi_i=\sqrt{\Delta_i}\phi_i$ and $\Delta_i=-\det\Hess_{p_i}(W_{q_0})$ (c.f. Proposition \ref{prop:identification-mir}).

By Corollary \ref{cor:B-S-flat}, consider the flat section of the quantum D-module $\cF$
\[
y^\theta_i:= (-2\pi z)^{-n/2}\sum_{j=1}^\fs \left(\int_{\Gamma_i^0}f_je^{-\frac{W_{q_0}}{z}}\right) H^j.
\]
By Equation \eqref{eqn:asymptotic-expansion}, $y_i^\theta$ has an asymptotic expansion 
\[
y_i^\theta e^{u_i/z}\sim(H^1,\dots,H^\fs)\begin{pmatrix}(\hat\phi_i,H_1)\\\vdots\\(\hat\phi_i,H_\fs)\end{pmatrix}(\mathrm{id}_{\bC^\fs}+O(z)),
\]
which precisely match Equation \eqref{eqn:Dubrovin-Givental-decomposition} (c.f. Equation \eqref{eqn:Psi}). So these flat sections $y_i^\theta$ are indeed analytic lift of the asymptotic solutions given by Theorem \ref{thm:Dubrovin-Givental-decomposition}.

We rotate $\gamma^0_i$ by a family $\gamma_i^t$, $0\leq t\leq 1$ like the following figure.

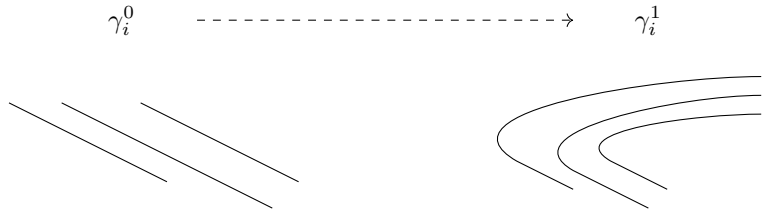
\[
\tikz{
\node at (-2,2.5) {$\gamma_i^0$};
\node at (5,2.5) {$\gamma_i^1$};
\draw [dashed, ->] (-1,2.5) to (4,2.5);
\begin{scope}[scale=0.7]
\draw (0,0) to  +(-4,2);
\draw (0.5, 0.5) to +(-3, 1.5);
\draw (-2,0.5) to +(-3, 1.5);
\end{scope}
\begin{scope}[shift={(5,0)}, scale=0.5]
\draw (0,0) to  +(-2,1) to[out=150, in =180]  (3,3);
\draw (0.5, 0.5) to +(-1.5, 0.75) to[out=150, in =180]  (3,2.5);
\draw (-2,0.5) to +(-1.5,0.75) to[out=150, in =180] (3, 3.5);
\end{scope}
}
\]

Let $\Gamma_i^t$ be the thimbles over $\gamma_i^t$. Effectively $\Gamma_i^t=e^{-\bi t \theta}\Gamma_i^0$. Then the resulting
\[
\bar y^\theta_i(z):=s_{[\Gamma_i^1]}.
\]
is an analytical continuation of $y_i^\theta(z)$, well-defined on $z>0$. These $\Gamma_i^1=\Gamma_i$ as in Definition \ref{def:thimbles-cycle}. By discussion in Section \ref{sec:thimbles-obj}, there exist $L_1,\dots, L_\fs\in \cWL$ and $[L_i]=[\Gamma_i]$ for all $i$. By Proposition \ref{prop:exceptional} they form an exceptional collection. By HMS Theorem \ref{thm:ccc} and \ref{thm:GPS}, there exists a full exceptional collection $E^\theta_1,\dots,E^\theta_\fs\in \Coh(X)$, such that each $L_i\cong \mu\circ\iota(E^\theta_i)$.

Therefore by Proposition \ref{prop:any-primary-insertion}
\begin{align*}
\bar y^\theta_i&=(-2\pi z)^{-n/2}\sum_{i=1}^\fs \left(\int_{\Gamma_i}f_je^{-\frac{W_{q_0}}{z}}\right) H^j=\sum_{j=1}^\fs\llangle H_j,\frac{z^{\widetilde{-\deg}}z^{c_1(X)} A_{E^\theta_i}}{z+\psi}\rrangle_{0,2}^X\vert_{\btau=\btau_0} H^j\\
&=\cZ([E_i^\theta])\vert_{\btau=\btau_0}.
\end{align*}
Then we have reached our conclusion.
\begin{theorem}
\label{thm:gamma-ii}
  The Gamma II conjecture (Conjecture \ref{conj:gamma-ii}) is true for a complete smooth Fano toric manifold at any point $\btau=\btau_0\in \tUeec$ for sufficiently small $\epsilon$ and $\epsilon'$, i.e. near the large radius limit.
\end{theorem}

\end{document}